\documentclass{amsart}
\usepackage{amsmath,amssymb,amsthm,color}
\usepackage[mathscr]{euscript}

\makeatletter
\@namedef{subjclassname@2020}{\textup{2020} Mathematics Subject Classification}
\makeatother

\def \R {{\mathbb R}}
\def \N {{\mathbb N}}
\def \Z {{\mathbb Z}}
\def \C {\mathcal{C}}
\def \d {\mathbf d}
\def \U {\mathbb U}
\def\T{\mathcal{C}}
\def\app{{`}}
\def \f {{\bar f}}

\newtheorem{theorem}{Theorem}[section]
\newtheorem{lemma}{Lemma}[section]
\newtheorem{proposition}{Proposition}[section]
\newtheorem{definition}[theorem]{Definition} 
\newtheorem{remark}[theorem]{Remark}

\author{Anna Chiara Lai}
\address{A. C, Lai, Dipartimento di Scienze di Base e Applicate per l'Ingegneria,
Sapienza Universit\`a di Roma, Via Scarpa 16, 00181, Roma, Italy\\
Telefax (39)(06)442401836,\,\, Telephone (39)(06)  49766555}
\email{anna.lai@sbai.uniroma1.it}
\author{Monica Motta}
\address{M. Motta, Dipartimento di Matematica,
Universit\`a di Padova\\ Via Trieste, 63, Padova  35121, Italy\\
Telefax (39)(49) 827 1499,\,\, Telephone (39)(49) 827 1368}  
\email{motta@math.unipd.it}

\title{Converse Lyapunov theorems  for control systems with unbounded controls } 
\thanks{This research is partially supported by the  INdAM-GNAMPA Project 2020 ``Extended control problems: gap, higher order conditions and   Lyapunov functions"
.}

\begin{document}
\maketitle






\begin{abstract}
In this paper we extend well-known relationships between global asymptotic controllability, sample stabilizability,  and the existence of a control Lyapunov  function   to  a wide class of control systems with unbounded controls, which includes control-polynomial systems. In particular, we consider open loop controls and  discontinuous  stabilizing feedbacks, which may be  unbounded  approaching the target, so that the corresponding trajectories may present a chattering behaviour. 
A key point of our results is to prove that  global asymptotic controllability, sample stabilizability,  and  existence of a control Lyapunov  function for these systems  or for an {\em impulsive extension} of them are equivalent.  \end{abstract}



\section{Introduction}
In this paper we extend classic  equivalence results between global asymptotic controllability to a set $\C$, sample stabilizability to   $\C$,  and the existence of a control Lyapunov  function   to  a 
 control system of the form 
\begin{equation}\label{Eintro}
\dot x(t)=f(x(t),u(t)) \ \text{a.e.}, 
  \qquad u(t)\in U  \ \text{a.e.},
\end{equation}
where the (unbounded) control set $U\subseteq\R^m$ is  a closed cone, the target set $\C\subset\R^n$ is closed  with  compact boundary,  and the function $f:\R^n\times U\to\R^n$  satisfies suitable growth assumptions  in the control variable, which include control-polynomial dependence (see hypothesis {\bf (Hg)} below).  The extension lies in the fact that, following  \cite{LMR16,LM20,LM19},  we consider notions of global asymptotic controllability and  sample stabilizability  which involve open loop controls $u\in L^\infty_{loc}$ and   locally bounded  feedback laws $K:\R^n\setminus\C\to U$ with possibly  $\limsup_{x\to\bar x\in\partial\C}|K(x)|=+\infty$, respectively. 

\vskip 0.2 truecm
The problems of  asymptotic controllability and  feedback stabilization to a point or to a   set 
of   control systems 
 that are nonlinear  (and for which linearization fails),  
 and their relationships with the existence of  control Lyapunov functions have been central topics in control theory since the 1980s.  It is now well-known that a smooth control Lyapunov function, which guarantees  the asymptotic controllability of the system,  may not exist and a  continuous stabilizing feedback fails in general to exist   either  (see \cite{Brockett,SS80,Art83,S83,Ry94,CR94,SS95}). Under suitable assumptions on  the dynamics function $f$ and on the control set $U$,  these problems were solved in \cite{CLSS} by the introduction of nonsmooth control Lyapunov  functions,  discontinuous feedback laws $K=K(x)$, and a ‘‘sample and hold" solution concept,  similar to that used in differential games in  \cite{KS88}. 
 
In this context, converse Lyapunov theorems have been established, showing how global asymptotic controllability (GAC), which is equivalent to sample stabilizability by a result in  \cite{CLSS},  implies the existence of continuous (see \cite{S83}), and actually locally Lipschitz and even semiconcave  control Lyapunov functions. The  latter property plays a a fundamental role in the  the explicit construction of stabilizing feedback strategies (see \cite{CLRS,R00,R02,MRS, KT00,KT04}).  We have limited ourselves to mentioning only a few key articles and those most related to the present work.
For a broader overview  we refer e.g.  to   \cite{C10,DES11,K15,Tsi12,Tsi16} and references therein. 
 
Two are the  key hypotheses in the above results: (i)  $f$ is continuous in $(x,u)$ and  Lipschitz continuous  in $x$ on compact subsets of $\R^n$ (or `uniform in distance to the set $\C$',  as in \cite{KT00}),  uniformly with respect to $U$; (ii) the vector field   $f(x,u)$ associated to   $u=u(t)$ or  $u=K(x)$ and  steering   trajectories of \eqref{Eintro}  to $\C$ in a uniform way, is  bounded  in  any  neighborhood of the target. Actually, it is usually assumed that these open-loop  and  feedback controls   are  themselves   bounded  for states  close to  $\C$.  
\vskip 0.2 truecm
 Our aim is to extend these results to a wider class of control systems where  conditions (i), (ii) above do not hold. This extension is  not achieved by refining the techniques  used in the case of classic  assumptions on $f$.  
 Rather, following an approach commonly adopted in optimal impulsive control,  as generalized in \cite{RS00,MS14,KDPS14},  we  embed the original control system into an extended control system  with bounded controls, to which the known results apply.  Our main result is that GAC, sample stabilizability,  and  existence of a control Lyapunov  function for the extended system or for the original system are all equivalent properties (see the Converse Lyapunov Theorem \ref{thm3}).
 These relationships between the original control system and its impulsive extension  are relevant in themselves. Indeed, on the one hand, it is not obvious a priori  that  asymptotic controllability or stabilizability   to the target by 
means of impulsive inputs  guarantees  controllability and  
stabilizability of  \eqref{Eintro} to  $\C$  over unbounded controls,  since 
 trajectories of the impulsive extension may not be approximated  by  trajectories of  \eqref{Eintro} with
the same endpoint. On the other hand, the explicit construction of control strategies for the (impulsive) extended system, which is usually simpler, can be used to build a stabilizing feedback also for the original control system  \eqref{Eintro}, as described in Section \ref{s4}. 
\vskip 0.2 truecm
More in detail, the main hypotheses we will consider on $f$, are:  
 \vskip 0.2 truecm
\noindent {\bf(Hg)} {\em  there exists a strictly increasing,  bijective function $\nu:[0,+\infty)\to[0,+\infty)$, to which we refer to as  {\em growth rate},  such that

{\bf (i)} the  function $\f:(\R^n\setminus\C) \times U\to\R^n$, defined by
			$
			\displaystyle \bar{f}(x,u) := \frac{ {f} (x,u)}{1 +\nu(|u|)}, 
			$
 is uniformly continuous on ${\mathcal K}\times U$  for  any compact  set ${\mathcal K}\subset \R^n\setminus\C$   and  bounded on $(B_R(\C)\setminus \C)\times U$ for any $R>0$;  		

{\bf (ii)}  the function  $F:\overline{(\R^n\setminus\C)}\times{[0,+\infty)\times U}\to \R^n$, defined  
  as
$$
F(x,w_0,w): =\lim_{r\to w_0^+}\f\left(x,\frac{w}{|w|}\,\nu^{-1}\left(\frac{|w|}{r}\right)\right)=\lim_{r\to w_0^+}f\left(x,\frac{w}{|w|}\,\nu^{-1}\left(\frac{|w|}{r}\right)\right) r \  \footnote{For any  $w\in\R^m$, when $w=0$ we mean that $\frac{w}{|w|}=0$.} 
$$
is not identically zero, it is continuous, and  locally bounded on  $\overline{(\R^n\setminus\C)}\times  \overline{{\U}}$, where 
$
 {\U}:=\{(w_0,w)\in(0,+\infty)\times U: \ \ w_0+|w|=1\}.
 $ 
 }
\vskip 0.2 truecm 
 The function  $\nu$ represents the  maximal  growth  of  $f$ in the control $u$.  We will refer to $\bar f$ and $F$  as  {\em rescaled dynamics}  and {\em extended dynamics} function, respectively.  This extension consists essentially in  a  control-compactification, obtained by  adding the scalar control $w_0$, so that the pairs $(w_0,w)=(0,w)$ with $|w|=1$ represent the points of $U$   `at infinity'.  Observe that control-polynomial systems of degree $d$  with   continuous coefficients satisfy  hypothesis {\bf(Hg)} with growth rate $\nu(r)=r^d$.   
\vskip 0.2 truecm 
Under  assumption {\bf(Hg)} and some standard Lipschitz continuity hypotheses on the extended dynamics $F$ specified in Section \ref{s4}, we establish our Converse Lyapunov Theorem.  The  proof of  this theorem relies on three key results: Theorem \ref{thmsigma}, dealing with two equivalent notions of GAC  for unbounded dynamics, which is crucial in order to prove that  GAC of  $\dot x=f(x,u)$ and GAC of the rescaled control system  $\dot x=\f(x,u)$ are equivalent  (see Theorems \ref{thm2}, \ref{ThEEr});  the interplay  between the rescaled and the extended control system (see Propositions \ref{p3}, \ref{p3}); and the fact that sample stabilizability implies GAC also for dynamics which are merely continuous on $\R^n\setminus\C$ (see Theorem \ref{PSGAC}). 

\vskip 0.2 truecm
Let us  point out  that there are interesting situations in which considering bounded controls (and dynamics)  the system is not asymptotically controllable and not stabilizable, whereas it becomes so   if dynamics   that can become unbounded when approaching the target are allowed  (see the  example in Section \ref{Sex}). This is the case, for instance, of some applications to Lagrangian mechanics  where part of the coordinates act as controls. The evolution of the remaining coordinates is then described  by an ``impulsive" control system, where the  dynamics function  is linearly or quadratically dependent on the derivatives of the controlled coordinates, derivatives which are identifiable  with unbounded controls  (see \cite{AB1,Mar91},  and  \cite{BR10} with references therein). In particular, in  \cite{BR10} the authors   exhibit mechanical examples  for which stabilization can only be achieved by ‘‘vibrating controls", namely   allowing unbounded inputs.
 \vskip 0.2 truecm
 The paper is organized as follows.  In Section \ref{s2} we introduce the notions of  GAC and  GAC with $U\cap\sigma$ controls,  prove that they are equivalent, and  show that  sample stabilizability implies  GAC.  Section \ref{s3} is devoted to establish some  relationships between the rescaled and the extended system.  In  Section \ref{s4}   we  prove the converse Lyapunov theorem and describe how related explicit feedback constructions for the original and for the extended control system can be implemented.  In Section \ref{Sex}, an example  concludes the paper.
 
	\subsection{Notations}\label{Sprel}  For $a,b\in\R$, we set $a\vee b:=\max\{a,b\}$, $a\land b:=\min\{a,b\}$.  Let $\Omega\subseteq \R^N$ for some integer $N\ge1$ be a nonempty set. For every $r\geq 0$, we set $B_r(\Omega):=\{x\in \R^n: \ d(x,\Omega)\leq r\}$, where $d$ is the usual Euclidean distance. We use $\overline\Omega$,  $\partial\Omega$, and  $\mathring \Omega$ to denote the closure, the boundary, and the interior of $\Omega$, respectively.  For any interval $I\subseteq\R$,    $L^\infty(I,\Omega)$,  $AC(I,\Omega)$ are the sets of  functions  $x:I\to\Omega$, which are essentially bounded or absolutely continuous, respectively,  on $I$. We use $L^\infty_{loc}(I,\Omega)$,  $AC_{loc}(I,\Omega)$ to denote the sets of  functions  $x:I\to\Omega$, which are essentially bounded or absolutely continuous  on any compact subset $J\subset I$. When no confusion may arise, we  simply write $L^\infty(I)$,  $AC(I)$, $L^\infty_{loc}(I)$,  $AC_{loc}(I)$.

\section{Global asymptotic controllability and sample stabilizability}\label{s2}
We  introduce two concepts of  global asymptotic controllability  and prove that they are equivalent.  Furthermore,  we show that sample stabilizability implies global asymptotic controllability,  as in the case of bounded controls.  
\vskip 0.2 truecm
Unless otherwise specified,   we assume $f:(\R^n\setminus\C)\times U\to\R^n$ continuous.  Let us set    $\mathbf d(x):=d(x,\C)$.
\subsection{Equivalent concepts of global asymptotic controllability}
\begin{definition}[Admissible  trajectory-control pair]\label{Admgen}   A couple $(x,u)$  is called an {\em admissible  trajectory-control pair}  for $(f,U)$   if   there exists $T_x\le +\infty$ such that:  the control  $u$   belongs to $L^\infty_{loc}([0,T_x),U)$;    the trajectory $x\in AC_{loc}([0,T_x),\R^n\setminus \C)$  verifies 
	 \begin{equation}\label{Egen}
	 \dot x(t)=f(x(t),u(t)) \qquad\text{a.e. $t\in[0,T_x)$;} 
	 \end{equation}
 and,  if $T_x<+\infty$, one has  $\lim_{t\to T_x^-}\d(x(t))=0$. 	
 If $(x,u)$ is an admissible trajectory-control pair   for $(f,U)$   and $T_x <+\infty$,   we extend $x$ to $[0,+\infty)$ by setting $x(t):= \bar z$ for any $t\geq T_x$,  where $\bar z$ is an arbitrary point of the set
\[
\C_x := \big\{ \zeta \in \partial\T \text{ :}  \ \exists\tau_i\uparrow T_x \,\,\,\text{as $i\to+\infty$ and } \lim_{i \to +\infty} x(\tau_i) =\xi \big\}.\footnote{By the very definition of $T_x$, we have that $\C_x\ne\emptyset$, as $\partial\T$ is assumed to be compact. Hence, for each admissible trajectory-control pair $(x,u)$  for $(f,U                                                                                                                                                                                                                                                                                                                                      )$, the trajectory $x$, possibly extended as above, is always defined on the whole interval $[0,+\infty[$.}
\]
\end{definition}
When no confusion may arise,  we will simply  call {\em admissible trajectory-control pair} any  admissible trajectory-control pair for $(f,U)$.
	
As customary, we   use   ${\mathcal{KL}}$ to denote the set of all continuous functions 
	$\beta:[0,+\infty)\times[0,+\infty)\to[0,+\infty)$ such that:
	(1)\, $\beta(0,t)=0$ and $\beta(\cdot,t)$ is strictly increasing and unbounded 
	for each $t\ge0$;
	(2)\, $\beta(r,\cdot)$ is strictly decreasing  for each $r\ge0$; (3)\, $\beta(r,t)\to0$ 
	as $t\to+\infty$ for each $r\ge0$.
	We  refer to any function $\beta\in{\mathcal{KL}}$   as a {\em descent rate}. 
	
	\begin{definition}[GAC] The system \eqref{Egen}  is called {\em Globally Asymptotically Controllable  (GAC)   to $\C$}  if there exists a function $\beta\in \mathcal{K}\mathcal{L}$ such that for any initial point $z\in\R^n\setminus \C$ there is an admissible trajectory-control pair $(x,u)$  for $(f,U)$ with $x(0)=z$,  such that
\begin{equation}\label{Dd}
\mathbf d(x(t))\leq \beta(\mathbf d(z),t) \quad \forall t\geq 0.
\end{equation}
\end{definition}
	\begin{definition}[GAC with $U\cap \sigma$ controls]\label{GACU}
Let $\sigma:(0,+\infty)\to (0,+\infty)$ be a continuous function. We say that the system \eqref{Egen} is  {\em Globally Asymptotically Controllable  (GAC)   to $\C$ with $ U\cap\sigma$ controls}  if there exists a descent rate $\beta\in \mathcal{K}\mathcal{L}$ such that for any initial point $z\in\R^n\setminus \C$ there is an admissible trajectory-control pair $(x,u)$ for  $(f, U)$ with $x(0)=z$, such that
$$
\mathbf d(x(t))\leq \beta(\mathbf d(z),t) \quad \forall t\geq 0
$$
and
\begin{equation}\label{GACdef2}
|u(t)|\leq \sigma(\mathbf d(x(t)) \quad  \text{for a.e. $t\in[0,T_x)$}.
\end{equation}
\end{definition}
\begin{remark} {\rm 
The concept of GAC with $U\cap \sigma$ controls   introduced above only apparently coincides with the definition considered, e.g.,  in the survey papers \cite{S99,KT00}. In fact, in the previous literature the function $\sigma$ was supposed to be increasing, in order to prevent unbounded inputs around the target.  On the contrary, in Definition \ref{GACU}  it may happen that  $\lim_{r\to0^+}\sigma(r)=+\infty$, thus allowing for  controls  with $L^\infty$ norm diverging to $+\infty$ as the trajectory approaches  $\C$. }
\end{remark}

The concept of GAC with $ U\cap \sigma$ controls, although apparently stronger than GAC, when $f$  is locally Lipschitz continuous in $x$,   is   equivalent to GAC.  
\vskip 0.2 truecm
Precisley, let us   consider the following hypothesis: 
 \vskip 0.2 truecm
\noindent   {\bf (Hl)} {\em the  function $f:(\R^n\setminus\C)\times U\to\R^n$ is continuous and,  for every pair of compact sets ${\mathcal K}\subset\R^n\setminus\C$, $U_1\subset U$, there is some constant $L>0$  such that  
$$
|f(x,u)-f(y,u)|\le L\,|x-y| \qquad \forall x,y\in{\mathcal K}, \ \forall u\in U_1.
$$}
\begin{theorem}\label{thmsigma} Let $f$ satisfy  {\bf (Hl)}. Then,  system \eqref{Egen} is GAC to $\C$ if and only if it is GAC to $\C$ with $ U\cap \sigma$ controls.
\end{theorem}
	\begin{proof}
	If system  \eqref{Egen} is GAC to $\C$ with $U\cap\sigma$,   it  is trivially  GAC to $\C$. So,  let us assume  that \eqref{Egen} is GAC to $\C$ and prove that it is GAC with $U\cap\sigma$ controls by  building a continuous  positive function $\sigma$ and a descent rate  $\bar \beta$.
\vskip 0.2 truecm	 
{\em Step 1.}  ($\beta$-admissible trajectory-control pairs) Without loss of generality, in the definition of GAC  we can assume  the inequality in \eqref{Dd}   strict, namely, that  there exists a descent rate  $\beta\in \mathcal{K L}$ such that for all $z\in \R^n\setminus \C$ there is some admissible trajectory-control pair $(x,u)$ with $x(0)=z$,    such that
	\begin{equation}\label{b}
	\d(x(t))<\beta(\d(z),t)\qquad \forall t\geq 0.
	\end{equation} 
	We refer to such $(x,u)$ as a $\beta$-\emph{admissible trajectory-control pair from $z$}.  Let us define the set  $\mathcal A_\beta$, given by the triplets $(x,u,z)$,  where $z\in \R^n\setminus \C$  and, for any $z$, we select one    $(x,u)$ among the admissible trajectory-control pairs from $z$.  Note that, by \eqref{b}, one has
	\begin{equation}\label{bin}
	\beta(R,0)>R\quad \forall R> 0. 
	\end{equation}
	\vskip 0.2 truecm
	{\em Step 2.}  ($\beta$-strips) Let $r_0:=1$ and recursively define $(r_i)_{i\in\Z}$ by  
	$$r_{i-1}=\beta(r_i,0)\qquad i\in\Z,$$
	so that, for instance,  $r_1$ is the solution of $\beta(r_1,0)=1=r_0$ and $r_{-1}=\beta(r_0,0)$. By \eqref{bin} and by the definition of $\mathcal{KL}$ functions, we have that $(r_i)_{i\in\Z}$ is strictly decreasing, positive and 
	$$\lim_{i\to +\infty} r_i= 0 \quad \text{ and }\quad \lim_{i\to -\infty} r_i=+\infty.$$ 
	For every $i\in\Z$,  set $\mathcal B_i:= \{z\in\R^n\setminus\C: \  \d(z)\in [r_i,r_{i-1}]\}$. We define the \emph{$i$-th $\beta$-strip} as the set {$\mathcal A^i_\beta:=\{(x,u,z)\in\mathcal A_\beta: \ \  z\in \mathcal B_i\}$. Note that for all $i\in \Z$,  $r_{i-2}=\beta(r_{i-1},0)$, therefore, for every  $(x,u,z)\in\mathcal A^i_\beta$ one has $\d(x(t))<r_{i-2}$ for all $t\geq 0$.}

	Fix $i\in \Z$ and consider a triplet $(x,u,z)\in \mathcal A^i_\beta$.  Define 
	\begin{equation}\label{tdefiz}T_{i,z}:=\inf \left\{t\geq 0: \  \d(x(t))=\frac{r_i+r_{i+1}}{2}\right\}.  \end{equation}
	Clearly,  $0<T_{i,z}<T_x$ and  the fact that $u\in L_{loc}^\infty([0,T_x))$ implies
	\begin{equation}\label{proof3}
\|u \|_{L^\infty([0,T_{i,z}])}<+\infty.
\end{equation}
Set
	$$\tilde \varepsilon_{i,z}:=\inf\left \{\frac{1}{2}(\beta(\d(z),t)-\d(x(t)): \   t\in [0,T_{i,z}]\right\}.$$
	Note that, by the continuity of $\beta$ and $x$, $\tilde \varepsilon_{i,z}$ is actually a minimum. Furthermore,  $\tilde \varepsilon$ is positive in view of \eqref{b}. Define
	$$\bar  \varepsilon_{i}:=\frac{r_i-r_{i+1}}{4}, \qquad \varepsilon_{i,z}:=\min \{\tilde \varepsilon_{i,z},\bar \varepsilon_i\}.$$

	In view of  the Lipschitzianity  hypothesis {\bf (Hl)}, there exists $\delta_{i,z}>0$ such that, for all $\bar z\in \R^n\setminus \C$ verifying $|z-\bar z|<\delta_{i,z}$,    the Cauchy problem $\dot x=f(x,u)$, $x(0)=\bar z$,   admits  a unique  solution, denoted in the following by  $x(\cdot\, ;u,\bar z)$, which is  defined on the whole interval $[0,T_{i,z}]$ and verifies
	\begin{equation}\label{proof1}
	|x(t)-x(t; u,\bar z)|<\varepsilon_{i,z} \quad \forall t\in[0,T_{i,z}].
	\end{equation}
	From the definition of $\varepsilon_{i,z}$,  it follows that 
	$$
	\begin{array}{l} 
	\displaystyle\d(x(t; u,\bar z))<\d(x(t))+\frac{1}{2}(\beta(\d(z),t)-\d(x(t))<\beta(\d(z),t) \\
	\ \quad\qquad\qquad\leq \beta(r_{i-1},0)=r_{i-2}\quad \forall t\in [0,T_{i,z}],
	\end{array}
	$$
	while  the definition of $T_{i,z}$ yields
	$$\d(x(t; u,\bar z))\geq \d(x(T_{i,z}))-\tilde \varepsilon_i>r_{i+1} \quad \forall t\in [0,T_{i,z}].$$
	In conclusion, for any   $(x,u,z)\in \mathcal A^i_\beta$ there exists some $\delta_{i,z}>0$ such that,  for all  $\bar z\in\R^n\setminus\C$ verifying $|z-\bar z|<\delta_{i,z}$, one obtains
	\begin{equation}\label{proof2}
	\begin{cases}\d(x(t; u,\bar z))\in (r_{i+1},r_{i-2})&\forall t\in[0,T_{i,z}],\\
	\bar x(T_{i,z})\in {\mathcal B}_{i+1}.\end{cases}
	\end{equation}
	
	\vskip 0.2 truecm
	{\em Step 3.}  (Construction of $\sigma$ and  $\bar \beta$) Fix $i\in \Z$ and consider the cover of $\mathcal B_i$ given by the collection of open balls $\mathring B_{\delta_{i,z}}(\{z\})$, with $z\in \mathcal B_i$. Since $\partial \C$ is compact, then the $i$-th strip $\mathcal B_i$ is compact, as well, and consequently it admits a finite subcover $\{ \mathring B_{\delta_{i,z}}(\{z\})\}_{z\in \mathcal Z_i}$, where $\mathcal Z_i$ is a finite subset of $\mathcal B_i$. Now, define 
	\begin{equation}
	\bar \sigma(i):=\max \{\|u \|_{L^\infty([0,T_{i,z}])}: \ \  (x,u,z)\in \mathcal A^i_\beta, \ \ z\in\mathcal Z_i\},
	\end{equation}
where $T_{i,z}$ is as in \eqref{tdefiz}, and $\sigma:(0,+\infty)\to (0,+\infty)$, given by
$$\sigma(r):=\max\{\bar \sigma(i-i),\bar \sigma (i),\bar \sigma(i+1)\} \qquad\forall r\in[r_i,r_{i-1}).
$$
Note that,  for every $i\in\Z$, one has
\begin{equation}\label{sigma}
\sigma(r)\geq \bar \sigma(i) \qquad \forall  r\in [r_{i+1},r_{i-2}).
\end{equation}
To build a new descent rate function, (to be associated to controls $u$ such that $|u(t)|\le\sigma(\d(x(t)))$ for every $t\ge0$), set
\begin{equation}\label{Ti}
T_i:=\max\{T_{i,z}: \  z\in \mathcal Z_i\}.
\end{equation}
Replacing any time  $T_i$ with a larger value if necessary, we can assume that, for every  $i\in\Z$, one has  $\sum_{j= 0}^{+\infty} T_{i+j}=+\infty$.
Then, for every $i \in \Z$ and  $N\in\N$, define 
\begin{equation}\label{Tisum}
\bar T_{i,-1}:=0, \qquad \bar T_{i,N}:=\sum_{j=0}^N T_{i+j}.
\end{equation}
Note that $\bar T_{i,N}\to +\infty$ as $N\to \infty$ for every fixed $i\in\Z$. 
Consider the piecewise constant function $b:[0,+\infty)\times[0,+\infty)\to[0,+\infty)$, given by
$$
\left\{\begin{split}
&b(R,t):=r_{i+N-2} \quad \text{if } R\in [r_i,r_{i-1}) \text{ and } t\in [\bar T_{i,N-1}, \bar T_{i,N}),\\
& b(0,t)=0 \qquad\qquad \forall t\geq 0,
\end{split}\right.
$$
for all $i\in\Z$ and $N\in\N$.
To make notation more compact, we introduce the decreasing, integer valued function $i(R):= i\in \Z$, such that $R\in[r_i,r_{i-1}) $. Note that $i(R)\to +\infty$ as $R\to 0$ and $i(R)\to-\infty$ as $R\to+\infty$. We then rewrite the definition of $b$ as follows
$$b(R,t):=r_{i(R)+N-2} \quad \text{ if } t\in [\bar T_{i(R),N-1},\bar T_{i(R),N}), $$
for $N\in\N$. 
Since $(r_i)$ vanishes as $i\to +\infty$, for all $R\in (0,+\infty)$, $b(R,t)\to 0$ as $t\to +\infty$. Moreover, since $r_{i(R)}\to +\infty$  as $R\to +\infty$, for every $t\geq 0$, we have that $b(R,t)\to+\infty$ as $R\to +\infty$. Then $b$ can be dominated by some $\mathcal {KL}$ function, (say, a larger, continuous linear interpolation) that we call $\bar \beta$.

\vskip 0.2 truecm
	{\em Step 4.} (GAC with $U\cap \sigma$ controls) To conclude, for every initial datum $\bar z\in \R^n\setminus\C$ we need to provide a $\bar \beta$-admissible trajectory control pair $(\bar x,\bar u)$ from $\bar z$,  satisfying $|\bar u(t)|\leq \sigma(\d(\bar x(t))$ for all $t\geq 0$. 

To this aim, fix $\bar z\in \R^n\setminus \C$ and, for brevity, set $i:=i(\mathbf d(\bar z))$, so that $\bar z\in \mathcal  B_i$. By Step 3 it follows that  $\bar z\in\mathring B_{\delta_{i,z_0}}(\{z_0\})$ for some $z_0\in\mathcal Z_i$ and, if $(x_0,u_0,z_0)\in\mathcal A^i_\beta$ (taking into account  also the inequality \eqref{sigma}), there exists $\hat t_0:=T_{i,z_0}\leq T_{i}$   such that the   trajectory $\bar x:=x(t; u_0,\bar z)$, satisfies
\begin{align}
&\d(\bar x(t))\in (r_{i+1}, r_{i-2}) \quad \forall t\in [0,\hat t_0]\label{a},\\
&\bar x(\hat t_0))\in \mathcal B_{i+1}\label{b1},\\
&|u_0(t)|\leq \bar \sigma (i)\leq \sigma(\d(\bar x(t)))\quad \forall t\in [0,\hat t_0]\label{c}.
\end{align}
 By repeating the same argument starting from  the point  $\bar z_1:=\bar x(\hat  t_0)\in \mathcal B_{i+1}$,  one obtains a time $\hat t_1\leq T_{i+1}$ and a control $u_1\in L^\infty([0,\hat t_1],U)$ which satisfy an updated version of \eqref{a}-\eqref{c}, with $i$ replaced by $i+1$. Iterating this procedure, one gets a sequence of times  $(\hat t_N)_N$ and a sequence of controls $(u_N)$, such that $\hat t_N\leq T_{i+N}$  and   $u_N\in L^\infty([0,\hat t_N],U)$ for every $N$. Therefore,  setting for every $N\in\N$,
 $$
 \begin{array}{l} 
  \hat T_{-1}:=0, \quad \hat T_{N}:=\sum_{j=0}^N\, \hat t_{j},  \quad \hat T_{\infty}:=\sum_{j=0}^{+\infty}\, \hat t_{j}, \\
     \hat u(t):=u_N(t-{\hat T_{N-1}})\quad \forall {t\in(\hat T_{N-1},\hat T_{N}]},
     \end{array}
  $$
  one obtains a control $\hat u\in L_{loc}^\infty([0,\hat T_\infty), U)$ such that the corresponding trajectory $\hat x:=x(t; \hat u,\bar z)$ is defined on the whole interval $[0,\hat T_\infty)$ and  enjoys the following properties: 
\begin{align}
&\d(\hat x(t))\in (r_{i+N+1}, r_{i+N-2}) \quad \forall t\in [\hat T_{N-1}, \hat T_{N}]\label{aN},\\
&\hat x(\hat T_N)\in \mathcal B_{i+N+1}\label{bN},\\
&|\hat u(t)|\leq \sigma(\d(\hat x(t)))\quad \forall t\in [0,\hat T_N]\label{cN}.
\end{align}
In particular,  $\d(\hat x(t))\to 0$ as $t\to\hat T_{\infty}^-$. 
 
At this point, a simple inductive argument shows that 
$$\d(\hat x(t))<r_{i+N-1} \quad \forall t \ge \hat T_{N}.$$
Furthermore, for every fixed $N$, $\hat T_N$, which depends on  $\bar z$,  is  bounded above by a constant which depends only on $\d(\bar z)$. Indeed,  by construction,  $\hat T_N \leq \bar T_{i,N}$, for all $N\geq0$. 
Hence, one finally obtains that
$$\d(\hat x(t))<r_{i(\d(z))+N-2}= b(\d(\bar z),t)\leq \bar \beta(\d(\bar z),t), \    t\in[\bar T_{i(\d(z)),N-1},\bar T_{i(\d(z)),N}).$$
Since $\bar T_{i(\d(\bar z)),N}\to +\infty$ as $N\to \infty$, this concludes the proof. 
 \end{proof}
 
 \subsection{Sample stabilizability}\label{s5}
The feedback counterpart  of a notion of GAC which involves admissible trajectory-control pairs $(x,u)$ with controls  $u$ in $ L^\infty_{loc}([0,T_x), U)$, requires necessarily to consider  locally bounded feedback functions $K:\R^n\setminus\C\to  U$, which may have
$
\displaystyle\limsup_{x\to\bar x\in\partial\C}|   K(x)|=+\infty.
$

In correspondence of such feedbacks, following  \cite{LM20} we adopt the notion of {\it sample stabilizability} below. 
\vskip 0.2 truecm
A {\em partition}  (of 
	$[0,+\infty)$) is a sequence $\pi=(t_k) $ such that
	$t_0=0, \quad t_{k-1}<t_k$ \, $\forall k\ge 1$,  and
	$\lim_{k\to+\infty}t_k=+\infty$.  The value  diam$(\pi):=\sup_{ k\ge 1}(t_{k 
	}-t_{k-1})$ will be called the {\em diameter} or the {\em sampling time} of the 
	partition  $\pi$.
	\begin{definition}[Sampling  trajectory-control pair]
		\label{Ssolgen}   Given a locally bounded
		feedback  ${K}:\R^n \setminus \C\to  U$, a partition $\pi=(t_k)$,  and a  point
		$z\in\R^n \setminus \C$, we call  {\em $\pi$-sampling  trajectory-control pair}  for $\dot x=f(x,u)$ {\em from $z$,} a pair $(x,u)$, where the  sampling  trajectory $x$ is a continuous function  defined by recursively solving   
		$$
		\dot x=  {f}(x(t), {K}(x(t_{k-1}))) \qquad  \text{a.e. } t\in[t_{k-1},t_k],~ (x(t)\in\R^n\setminus \C)
		$$
		from the initial time $t_{k-1}$ up to time 
		$$
		\tau_k:=t_{k-1}\vee\sup\{\tau\in[t_{k-1},t_k]: \ x \mathrm{ \  is \, defined \,  on} \,    [t_{k-1},\tau)\},
		$$
		such that $x(t_0)=x(0)=z$. In this case, the   trajectory $x$ is defined on the right-open interval from time zero up to time
		$T^-:=\inf\{\tau_k: \ \tau_k<t_k\}$.  Accordingly,  for every $k\ge 1$ and  
		for all $t\in[t_{k-1},t_k)\cap[0,T^-)$, the {\em  sampling control} is defined as
		\begin{equation}\label{olcres}
		u(t):=  {K}(x(t_{k-1})) \quad \forall t\in[t_{k-1},t_k)\cap[0,T^-), \quad
		k\ge1.  
		\end{equation}
If   $T^-=T_x<+\infty$  such that  $\lim_{t\to T^-_x}\d(x(t))\to 0$,   we extend $x$ to $[0,+\infty)$ as described in Definition    \ref{Admgen}.	\end{definition}

	\begin{definition}[Sample stabilizability] \label{sstabgen}  A locally bounded   feedback 
		${K}:\R^n\setminus\C\to U$ is said to {\em sample stabilize  the control system  $\dot x= {f}(x,u)$  to $\C$}  
		if there is a  descent rate
		$\beta\in{\mathcal {KL}}$  satisfying the following: for each pair $0<r<R$  
		there exists  $\delta=\delta(R,r)>0$,  such that, 
		for every partition $\pi$ with  $\text{diam}(\pi)\le\delta$ and for any 
		$z\in\R^n\setminus\C$ such that $\d(z)\le R$, any  $\pi$-sampling  trajectory-control pair $(x,u)$   with $x(0)=z$ 
		is admissible and verifies: 
		\begin{equation}\label{betaS}
		\d(x(t))\le\max\{\beta(\d(z),t), r\} \qquad \forall t\in[0,+\infty).
		\end{equation}
		We call $\dot x= {f}(x,u)$  {\em sample stabilizable to $\C$} if there is a  feedback  $ {K}$ as above.
		
	\end{definition}

\begin{remark} {\rm 
Given a discontinuous feedback $K$, in this paper we only consider sampling trajectories, which   are classical solutions corresponding to piecewise constant controls. We just point out that, because of the mere continuity of   $f$ and the unboundedness of  $K$, 
 sampling trajectories  can have  a finite blow-up time  and chattering phenomena may  occur. As a consequence, classical  Euler solutions  --defined in  \cite{CLRS} as uniform limits of sampling solutions-- may not exist.  For this reason, in  \cite{LM20} (see also \cite{LM19}) we proposed a notion of \emph{weak} Euler solution, given by the pointwise limit of a sequence of suitably truncated sampling trajectories.  In particular, in   \cite{LM20} it has been shown that sample stabilizability  implies  weak Euler stabilizability.  }
\end{remark}

The main result of this subsection is:
\begin{theorem}\label{PSGAC}
Let   $f:(\R^n\setminus\C)\times U\to\R^n$ be continuous. Then, if the  control system  $\dot x=f(x,u)$ is sample stabilizable to $\C$, it is GAC to $\C$.
\end{theorem}
\begin{proof} 
Assume that   $\dot x={f}(x,u)$   is sample stabilizable to $\C$.  Let $K:\R^n\setminus\C\to U$ be a locally bounded,  sample stabilizing feedback, let $\beta\in{\mathcal {KL}}$ be an associated descent rate, and,  for any $r,R\in(0,+\infty)$ with $r<R$, let  $\delta(R,r)>0$ be as in Definition \ref{sstabgen}. 

Consider  the  sequence $(r_i)_{i\in\Z}$  introduced in  Step 2 of the proof of Theorem \ref{thmsigma}, defined in a recursive way by setting 
$$r_0:=1, \qquad    
r_{i-1}=\beta(r_i,0) \quad \forall i\in\Z.
$$
As already observed, this sequence  is  positive, strictly decreasing, and satisfies $\displaystyle\lim_{i\to-\infty}r_i=+\infty$,  $\displaystyle\lim_{i\to+\infty}r_i=0$.
For every $i\in\Z$, let $\mathcal B_i:= \{z\in\R^n\setminus\C: \  \d(z)\in (r_i,r_{i-1}]\}$ 
and choose  a positive sequence $(\hat t_i)_{i\in\mathbb Z}$ such that $\beta(r_{i-1},\hat t_{i})\le r_{i}$ for all $i\in \mathbb Z$, and also satisfying
$$
\sum_{n=0}^{+\infty}\hat t_{i+n}=+\infty.
$$  
 For each $i\in\Z$, set  $\hat T_{i,-1}:=0$ and  $\hat T_{i,N}:=\sum_{n=0}^N\hat t_{i+n}$ for any $N\in\N$.
 Hence, define the piecewise constant function $b:[0,+\infty)\times [0,+\infty)\to[0,+\infty)$, given by
 $$b(R,t):=\begin{cases}
 r_{i-2+N} &\text{if }R\in[r_{i},r_{i-1}), t\in [\hat T_{i,N-1},\hat T_{i,N})\\
 0 &\text{if }R=0, t\geq 0,
 \end{cases} 
$$
 for all $i\in\Z$ and $N\in\N$. 
Note that if $t\in[0,\hat t_i)\subseteq [\hat T_{i,-1},\hat T_{i,0})$ then
\begin{equation}\label{bes}
b(r_{i},t)=r_{i-2}=\beta(r_{i-1},0)\geq \beta(r_{i-1},t),\end{equation}
for all $i\in\Z$. 
As observed in the proof of Theorem \ref{thmsigma},  we can approximate this function $b$ with a $\mathcal{K}\mathcal {L}$ function $\bar \beta$ such that   $\bar \beta(R,t)\geq b(R,t)$ for all $(R,t)\in[0,+\infty)\times[0,+\infty)$.

Now, fixed $i\in\Z$,  define $\delta_i:=\delta(r_{i},r_{i-1})$ and consider the partition $\pi_i:=(t_{i,k})_{k\geq 0}$ where $t_{i,k}:=k \delta_{i}$. For any $z\in  \mathcal B_i$,
select  a $\pi_i$ sampling-trajectory $x_i(t;z)$  from $z$ associated to  the sample stabilizing feedback $K$.
Then, by the above definitions, 
$$\mathbf d(x_i(t;z))\leq \min\{\beta(\d(z),t),r_{i}\} \quad \text{for } t\geq 0.$$
In particular, also in view of \eqref{bes},  we have
\begin{equation}\label{low}\mathbf d(x_i(t;z))\leq  \beta(r_{i-1},t)\leq b(r_{i},t)\leq \bar \beta(r_i,t)\leq  \bar \beta(\d(z),t)\  \text{for } t\in[ 0, \hat t_{i}].\end{equation}
and, since $ r_{i}\geq \beta(r_{i-1},t)$ for all $t\geq \hat t_i$, 
\begin{equation}\mathbf d(x_i(t;z))\leq r_{i} \quad \text{for } t\geq \hat t_{i}.\label{end}\end{equation}
 Let $u_i(t)$ be the sampling control associated to $x_i(t;z)$.
 \medskip

Consider the map $i:\R^n\setminus\C\to \Z$,   defined as $i(z):=i$ whenever $z\in \mathcal B_{i}$.  Fix $z\in \R^n\setminus\C$. Let us   build in a recursive way   an increasing sequence of times $\{T_N\}_{N\geq 0}$ such that $T_N\leq \hat T_{i{(z)},N}$  for all $N\geq 0$, and a trajectory-control pair $(x,u)$ defined in $[0,T_N]$ such that
$$\mathbf d(x(t))\leq \bar \beta (\mathbf d(z), t) \qquad \forall t\in[0, T_{N}],\qquad \mathbf d(x(T_N))=r_{i(z)+N}. $$ 
Precisely, for $N=0$, we define
$T_0:=\inf\{t>0: \ x_{i(z)}(t;z)\in \mathcal B_{i(z)+1}\}$. 
 Note that, in view of \eqref{end}, $T_0\leq \hat t_{i(z)}=\hat T_{i(z),0}$. Set
$$x(t):=x_{i(z)}(t;z),\qquad u(t):=u_{i(z)}(t)\qquad  t\in[0, {T_0}].
$$
From \eqref{low} and from the definition of $T_0$ we  derive that
$$\mathbf d(x(t))\leq \bar \beta (\mathbf d(z), t) \qquad \forall t\in[0, T_0];  \qquad \mathbf d(x(T_0))=r_{i(z)}.$$
Let now $N>0$ and let be defined $T_0,\dots,T_{N-1}$ (satisfying $T_n\leq\hat T_{i(z),n}$ for all $n=0,\dots, N-1$) and  a trajectory-control pair $(x,u)$ on $[0,T_{N-1}]$ satisfying  
$$\mathbf d(x(t))\leq \bar \beta (\mathbf d(z), t), \qquad \forall t\in[0, T_{N-1}],\qquad \mathbf d(x(T_{N-1}))=r_{i(z)+N-1} .$$
 Set $z_N:=x(T_{N-1})$ and observe that $z_N\in \mathcal B_{i(z)+N}$.  Define $ {t_N}:=\inf\{t>0: \ x_{i(z)+N}(t; z_N)\in \mathcal B_{i_0+N+1}\}$. 
Then, in view of \eqref{end}, $ {t_N}\leq \hat t_{i(z)+N}$. Set $T_N:=T_{N-1}+t_N$ and note that $T_N\leq \hat T_{i(z),N}$. We extend the definition of $(x,u)$ to  {$(T_{N-1},T_N]$} as follows
$$x(t)=x_{i(z)+N}(t-T_{N-1};z_N),\quad u(t)=u_{i(z)+N}(t-T_{N-1}),\quad  t\in (T_{N-1},T_N].
$$
 From \eqref{low} and \eqref{bes},  we  deduce that
\begin{equation}\label{step}\mathbf d(x(t))=\mathbf d(x_{i(z)+N}(t-T_{N-1};z_N)) \leq r_{i(z)+N-1},\qquad  t\in {(T_{N-1},T_N]}.\end{equation}
 On the other hand, one has
  $$r_{i(z)+N-1}=b(\mathbf d(z), t)\leq \bar \beta(\mathbf d(z), t), \qquad   t\in[\hat T_{i(z),N-1},\hat T_{i(z),N}].$$
Since $\bar \beta(\d(z),\cdot)$ is decreasing, then  $r_{i(z)+N-1}\leq\bar \beta(\d(z),t)$ for all $t\in[0,\hat T_{i(z),N}]$.
In particular,  $r_{i(z)+N-1}\leq\bar \beta(\d(z),t)$ for all $t\in[0,T_N]$, because $T_N\leq \hat T_{i(z),N}$. This, together with \eqref{step}, implies that
\begin{equation}\label{step2}\mathbf d(x(t)) \leq \bar \beta (\d(z),t) \qquad \forall t\in {(T_{N-1},T_N]}.\end{equation}
Moreover,  we have by construction $\mathbf d (x(T_N))=r_{i(z)+N}$.

So far, we iteratively constructed an admissible trajectory-control pair $(x,u)$ from $z$, which is  defined in $[0,\tilde T)$, where $\tilde T:=\lim_{N\to \infty}T_N$
 ($\le+\infty$),  $\mathbf d(x(T_N))=r_{i(z)+N}$ for all $N\ge0$, and  such  that 
$$\mathbf d(x(t)) \leq \bar \beta (\d(z),t) \qquad \forall t\in[0,\tilde T).$$
 Therefore,  $\mathbf d(x(t))\to 0$ as $t\to \tilde T^-$  and, extending  $x$ to $[0,+\infty)$ as in Definition \ref{Admgen} when $ \tilde T<+\infty$,  this yields
$$\mathbf d(x(t)) \leq \bar \beta (\d(z),t) \qquad \forall t\geq 0,$$
so concluding  the proof. 

\end{proof}

It is worth noting that the results of Theorems \ref{thmsigma},  \ref{PSGAC} are constructive, in the sense that, given a decrease rate $\beta$ associated with the GAC  with $U\cap\sigma$ controls or the sample stabilizability, respectively, we explicitly indicate how to obtain a decrease rate for the GAC.  In addition, thanks to Theorem  \ref{PSGAC}, sufficient conditions for GAC (in the case of unbounded control systems)  obtained in \cite{MR13,LMR16}, follow now as corollaries by  the sample stabilizability results in \cite{LM18,LM20}.

\section{The rescaled system and the impulsive extension}\label{s3}
In this section we establish some relationships between the global asymptotic controllability to $\C$ of a rescaled  control system and the associated impulsive extension. These results will be crucial to obtain the Converse Lyapunov Theorem of Section \ref{s4}.
 
\subsection{GAC  of the rescaled problem}\label{Sres}
Throughout this subsection,  the function $f:(\R^n\setminus\C)\times U\to\R^n$ is continuous and satisfies the  growth assumption {\bf (Hg),\,(i)},  for some growth rate $\nu$. Let   $\f$ denote  the associated rescaled dynamics.
\vskip 0.2 truecm	
For the purpose of distinguishing the admissible trajectory-control pairs of $(f,U)$  from those of $(\f,U)$, we will denote by $(x,u)$ the former and by $(y,v)$ the latter. 
Precisely, we simply say that  $(x,u)$ is an {\em admissible trajectory-control pair} when  $u\in L^\infty_{loc}([0,T_x),U)$, $x\in AC_{loc}([0,T_x),\R^n\setminus\C)$, and $x$ solves the {\em original control system}
	\begin{equation}\label{E}
		\dot x(t) = f(x(t),u(t))   \quad  \text{a.e. } t\in(0,T_x),
	\end{equation} 
	where $\lim_{t\to T^-_x}\d(x(t))=0$ whenever $T_x<+\infty$.  We call $(y,v)$ an {\em admissible rescaled trajectory-control pair}  when $v\in L^\infty_{loc}([0,S_y),U)$, $y\in AC_{loc}([0,S_y),\R^n\setminus\C)$, and $y$ solves the {\em rescaled control system}
	\begin{equation}\label{Er}
		y'(s) = \f(y(s),v(s))   \quad  \text{a.e. } s\in(0,S_y),\ \footnote{ In \eqref{Er} we use the apex \app\app\,$'$\,"  to  denote differentiation with respect to the new parameter $s$, in  order to stress that it does not coincide, in general, with the time variable  $t$, of \eqref{E}, as clarified by Lemma \ref{cres}.}
	\end{equation} 
	where $\lim_{t\to S^-_y}\d(y(s))=0$ whenever $S_x<+\infty$.
 When $T_x$ [$S_y$] is finite, we mean that   $x$  [$y$] is extended to $[0,+\infty)$ as described in Definition \ref{Admgen}.  
 \vskip 0.2 truecm
As an easy consequence of the chain rule,  admissible  rescaled trajectory-control pairs $(y,v)$ are in one-to-one correspondence with  admissible trajectory-control pairs $(x,u)$ through a time-change. 
	\begin{lemma}\label{cres}   Assume  $f$ continuous and satisfying  {\bf{(Hg),\,(i)}}. Fix $z\in\R^n\setminus\C$. 
	
		 {\rm (i)} Given an admissible process  $(x,u)$  from $z$,  set 
		$$
		\displaystyle s(t) :=  \int_0^t \left(1+ \nu(|u(\tau)|)\right)d\tau\, \  \forall t\in[0,T_x),  \ \  S_y:= \lim_{t\to T_x^-}s(t), \quad t(\cdot):=s^{-1}(\cdot).
		$$
		Then $(y, v)(s):=(x,u) \circ t(s)$,   $s\in[0,S_y)$,    is an admissible rescaled trajectory-control pair from $z$.
		
		{\rm  (ii)}  Vice-versa, let  $(y,v)$ be an admissible   rescaled trajectory-control pair from $z$ and set
		$$
		t(s):=\int_0^s (1+\nu(|v(\sigma)|)^{-1} d\sigma \ \ \forall s\in [0,S_{y }),  \ \   T_x:= \lim_{s\to S_y^-}t(s),\quad s(\cdot):=t^{-1}(\cdot).
		$$
		Then,  $(x, u)(t):=(y, v)\circ s(t)$, $t\in[0,T_x)$,   is an admissible  trajectory-control pair from $z$.
	\end{lemma}

The following theorem establishes the equivalence between the GAC to $\C$ with $U\cap \sigma$ controls of the original and the rescaled control system.
{\begin{theorem}\label{thm2} Assume  $f$ continuous and satisfying  {\bf{(Hg),\,(i)}}. Then,  the original control system \eqref{E} is GAC to $\C$ with $U\cap \sigma$ controls if and only if the rescaled control system \eqref{Er} is GAC to $\C$ with $U\cap \sigma$ controls (for the same $\sigma$). 
\end{theorem}

\begin{proof}
 Suppose  first that the rescaled control system  \eqref{Er} is GAC to $\C$ with $U\cap \sigma$ controls,  for some continuous function $\sigma:(0,+\infty)\to(0,+\infty)$ and some descent rate $\beta\in \mathcal K\mathcal L$. Hence, for every   $z\in\R^n\setminus \C$  there is an admissible rescaled trajectory-control pair $(y,v)$   such that $y(0)=z$ and
$$
\mathbf d(y(s))\leq \beta(\mathbf d(z),s), \quad |v(s)|\leq \sigma(\mathbf d(y(s)) \quad \forall s\geq 0.
$$
Consider now $(x,u)(t):=(y,v)\circ s(t)$ where $s(t)$ is the time-change introduced in Lemma \eqref{cres},(ii). Since $s(t)\geq t$ for all $t\geq 0$,  from the monotonicity properties of $\beta(r,\cdot)$ it follows  that
\begin{equation*}\mathbf d(x(t))=\mathbf d(y(s(t)))\leq \beta(\mathbf d(z),s(t))\leq \beta(\mathbf d(z),t) \quad \forall t\geq 0.\end{equation*}
Thus,  $\beta$ is a descent rate also for the original control system \eqref{E}. Moreover, 
\begin{equation*}
|u(t)|=|v(s(t))|\leq \sigma(\mathbf d(y(s(t)))=\sigma(\mathbf d(x(t)) \quad \forall t\in[0,T_x), \end{equation*}
where $T_x=\lim_{s\to S_y^-}t(s)$,  as in Lemma \eqref{cres},(ii). 

\vskip 0.2 truecm To prove the converse implication, let us assume that the original control system  \eqref{E} is GAC to $\C$ with $U\cap \sigma$ controls,  for some continuous function $\sigma:(0,+\infty)\to(0,+\infty)$ and some descent rate $\beta\in \mathcal K\mathcal L$. 

For any couple $(R,r)$ of  real numbers such that $0<r< R$, let us set
$$
N(R,r):= \nu\big(\max \sigma([r,\beta(R,0)]\big).
$$
By the  continuity of $\nu$, $\beta$, and $\sigma$ and by the monotonicity properties of  $\nu$ and $\beta$, one  deduces immediately that the function $N:\{(R,r): \   0<r< R\}\to (0,+\infty)$ is continuous,  for any $R$,   $r\mapsto  N(R,r)$ is decreasing,  and, for any $r$,   $R\mapsto  N(R,r)$ is increasing. By eventually enlarging $N$, we can assume without loss of generality that $N$ is strictly monotone  with respect to both $r$ and $R$. 
 
 Now, let $S(R,r)>0$ be the value of $s$  implicitly defined by the equation
			$$
			 \beta\left(R,\frac{s}{1+N(R,r)}\right)=r. 
			$$
			From the monotonicity  and continuity  properties of $\beta$ and $N$, it follows that    $S$ is a continuous function on $\{(R,r): \ 0<r<R\}$,  such that 
			    $r\mapsto S(R,r)$ is strictly decreasing   and $R\mapsto S(R,r)$ is strictly increasing. As a consequence, if, for any $R$,  we denote by $\rho=\rho(R,s)$ the inverse of the map $\rho\mapsto S(R,\rho)$,  one easily obtains  that  $\rho$ is a ${\mathcal{KL}}$ function. Moreover,  we have the identity
			    \begin{equation}\label{rho}
			    \beta\left(R,\frac{s}{1+N(R,\rho(R,s))}\right)=\rho(R,s).
			    \end{equation}
			 
			Let us show that $\rho$ is a descent rate for the rescaled control system \eqref{Er}. 
			 To this aim, fix $z\in\R^n\setminus\C$ and let $(x,u)$ be an admissible trajectory-control pair of \eqref{E} with $x(0)=z$ and satisfying 
\begin{equation}\label{GACo1}\mathbf d(x(t))\leq \beta(\mathbf d(z),t), \qquad  |u(t)|\leq \sigma(\mathbf d(x(t)) \quad \forall t\geq 0.\end{equation}
Now, define $(y,v)(s):=(x,u)\circ t(s)$ for all $s\in[0,S_y)$,  where
		 $$
		\displaystyle s(t) :=  \int_0^t \left(1+ \nu(|u(\tau)|)\right)d\tau\, \  \forall t\in[0,T_x),  \quad S_y:=\lim_{t\to T_x^-}s(t), \quad  t(\cdot):=s^{-1}(\cdot).
		$$
By Lemma \ref{cres} $(y,v)$  is an admissible trajectory-control pair for $(\bar f,U)$. Furthermore,  the rescaled control $v$ verifies
\begin{equation*}
|v(s)|=|u(t(s))|\leq \sigma(\mathbf d(x(t(s)))= \sigma(\mathbf d(y(s))\quad \forall s\in [0,S_y).
\end{equation*}
 Let us show that $\d(y(s))\le\rho(\d(z),s)$ for every $s\in[0,S_y)$.  
By contradiction, suppose that there exists some  $s\in[0, S_y)$
such that
\begin{equation}\label{contr}\mathbf d(y(s))>\rho(\mathbf d(z),s).\end{equation}
Since $y(s)=x(t(s))$,  
 we deduce that
$$\rho(\mathbf d(z),s)< \mathbf d(x(t(s))\leq \beta(\mathbf d(z),0).$$
 By the definitions of $N$,  $t(s)$,  and the monotonicity properties of $\beta$,  we get
\begin{equation}\label{contrat}
t(s)\geq \frac{s}{1+N(\mathbf d(z),\rho(\mathbf d(z),s))}.
\end{equation}
Using \eqref{GACo1} and \eqref{rho} we obtain the required contradiction with \eqref{contr}. Indeed, \eqref{contrat} and the monotonicity of $\beta$, imply that
$$
\begin{array}{l}
\d(y(s))=\d(x(t(s)))\leq \beta(\mathbf d(z),t(s))\leq \beta\left(\mathbf d(z),\frac{s}{1+N(\mathbf d(z),\rho(\mathbf d(z),s))} \right) \\
\qquad\qquad\qquad\qquad\qquad\qquad\qquad\qquad =\rho(\mathbf d(z),s).
\end{array}
$$
At this point, if  $S_y=+\infty$ the proof is concluded. If instead  $S_y<+\infty$, by definition we extend the trajectory $y$  to $[0,+\infty)$ as described in Definition \ref{Admgen}, so $\d(y(s))=0<\rho(\mathbf d(z),s)$ for every $s\ge S_y$ trivially.
\end{proof}
From Theorems  \ref{thmsigma}, \ref{thm2} one  derives  that, under the Lipschitz continuity assumption  {\bf(Hl)},    the original system is  GAC to $\C$  if and only if  the rescaled system is GAC to $\C$. Precisely, one has:
\begin{theorem}\label{ThEEr} Assume that $f$ satisfies  {\bf(Hl)} and {\bf (Hg),\,(i)}. Then, the original system \eqref{E} is GAC to $\C$ if and only if the rescaled system \eqref{Er} is GAC to $\C$. In addition, for both systems being GAC to $\C$  is equivalent to being GAC to $\C$ with $U\cap \sigma$ controls.
\end{theorem}
\begin{proof}  Theorem \ref{thm2} establishes  that the original system \eqref{E} is GAC to $\C$  with $U\cap \sigma$ controls if and only if the rescaled system \eqref{Er} is GAC to $\C$ with $U\cap \sigma$ controls. To conclude the proof, it is enough to observe that, when $f$ verifies  {\bf(Hl)}  and {\bf (Hg),\,(i)},  the rescaled dynamics $\f$   verifies  {\bf(Hl)} too. Hence, Theorem  \ref{thmsigma} applied both to $f$ and $\f$  implies the  equivalence of the notions of GAC to $\C$ and  GAC to $\C$  with $U\cap \sigma$ controls both for  the original  system and the rescaled system. 
\end{proof}

 \subsection{GAC of the impulsive extension}\label{Simp}
Let $f$ be a continuous function satisfying {\bf (Hg)}, and let   $\nu$, $\f$, $F$, $\U$ be  as in {\bf (Hg)}. We now  embed the original system into  an extended control system and show that  GAC to $\C$ of the rescaled control system implies GAC to $\C$ of the extended  system.  
\begin{definition}[Admissible extended trajectory-control pairs]\label{Dext} A triple $(y,w_0,w)$  is called an {\em admissible extended  trajectory-control pair}    if   there exists $S_y\le +\infty$ such that:  the control  $(w_0,w)\in L^\infty([0,S_y), \overline{\U})$;    the trajectory $y\in AC([0,S_y),\R^n\setminus \C)$  is a solution of the   {\em extended control  system} 
  \begin{equation}\label{Es}
y'(s) = F(y(s), w_0(s),w(s)) \qquad\text{a.e. $s\in[0,S_y)$;}
\end{equation}
and,  if $S_y<+\infty$, one has $\lim_{s\to S_y^-}\d(y(s))=0$. 	
 If $(y,w_0,w)$ is an admissible extended trajectory-control pair  and $S_y <+\infty$,   we extend $y$ to $[0,+\infty[$ by setting $y(s):= \lim_{\sigma\to S_y^-}y(\sigma)$ for any $s\geq S_y$.\footnote{In this case, the limit at  $S_y$ always exists, since $F$ is bounded on any neighborhood of the target, which has compact boundary.}
\end{definition}

The original and the rescaled system can be embedded in the extended system as follows. 
 \begin{lemma}\label{L1ext} Assume that $f$ is a continuous function satisfying   {\bf (Hg)} for some growth rate $\nu$.  Let  $(x,u)$ be  an admissible trajectory-control pair  for  $(f,U)$.  Set 
   $$
\begin{aligned} 
\displaystyle s(t) &:=  \int_0^t \left(1+ \nu(|u(\tau)|)\right)d\tau\, \  \forall t\in[0,T_x),  \ \  S_y:= \lim_{t\to T_x^-}s(t), \quad t:=s^{-1},\\   
y(s)& :=x \circ t(s) \ \  \forall s\in[0,S_y), \\   
 (w_0,w)(s)&:=\displaystyle\left(1 ,  \frac{u\,\nu(|u|)}{|u|}\circ t(s) \right)\,t'(s) = \left(\frac{1}{1+\nu(|u|)}\, , \,  \frac{u\,\nu(|u|)}{|u|(1+\nu(|u|))} \right)\circ t(s)
\end{aligned}
$$
for a.e.   $s\in[0,S_y)$. Then $(y, v)$, where  $v(s):=  u \circ t(s)$  for a.e. $s\in[0,S_y)$,    is an admissible rescaled trajectory-control pair,  while  $(y,w_0,w)$    is an admissible extended trajectory-control pair, with  $w_0>0$   a.e.  on $[0,S_y)$.
 \end{lemma}
\begin{proof}  The rescaled trajectory-control pair  $(y, v)$ is admissible by Lemma \ref{cres}, (i).  In view of the  definition of $(w_0,w)$ and of $F$,  straightforward   calculations yield that
\begin{equation}\label{vw}
y'(s)=\f(y(s),v(s))=F(y(s),w_0(s),w(s)), \quad \text{a.e. $s\in[0,S_y)$.}
\end{equation}
This shows that  $(y,w_0,w)$ is an admissible extended trajectory-control pair with $w^0>0$ a.e.. 
\end{proof}

\begin{remark}\label{rmkrecession}
{\rm  When   $\nu(r)=r^{\,\bar d}$ for some integer $\bar d\ge 1$,  the   extended dynamics   $F$   is equivalent to the extended dynamics  introduced in \cite{RS00}, whose definition is based on the notion of   {\em recession function}. Indeed, if  for any   $(x,w_0,w)\in  \overline{(\R^n\setminus\C)}\times  \overline{{\U}}$ we set $(\tilde w_0,\tilde w):=\left( w_0^{1/\,\bar d}, \frac{w}{|w|}\, |w|^{1/\,\bar d}\right)$, then $\tilde w_0^{\,\bar d}+ |\tilde w|^{\,\bar d}=1$ and  we obtain that 
$$
\begin{array}{l}
\displaystyle F(x,\tilde w_0,\tilde w)=\lim_{r\to \tilde w_0^+}\f\left(x,\frac{\tilde w}{r}\right)=\lim_{r\to \tilde w_0^+}f\left(x,\frac{\tilde w}{r}\right)\, r^{\,\bar d}.
\end{array}
$$
In particular, in the case of control-affine $f$ one has   $\bar d=1$ and   $(\tilde w_0,\tilde w)\equiv(w_0,w)$, so that  the classical  impulsive extension of the graph-completion approach  considered in \cite{LM19},  coincides with the present one. }
\end{remark}
The extension consists  in considering   $(y,w_0,w)$,  where $w_0$ may be zero on  nondegenerate subintervals of $[0,S_y)$. On these intervals, the time variable $t=\int_0^sw_0(\sigma)\,d\sigma$ is constant  --i.e. the time  stops--, while the state variable $y$ evolves, according to the equation  $y'=F(t,y,0,w)$,  sometimes called the {\em fast dynamics}. For this reason,  system \eqref{Es} is  often referred to  as the {\em impulsive extension} of the original control system \eqref{E}, despite the fact that it is an ordinary control system, as the extended controls  $(w_0,w)$ take values in the compact set $\overline{\U}$ and the trajectory $y$ is absolutely continuous.   A detailed discussion of this topic  goes beyond the purposes of the paper. We just mention  that an equivalent, $t$-based description of this extension, where $u$ is no more a function and the trajectory $x$ is  a discontinuous map whose   total variation is bounded on $[0,T]$ for every $T<T_x$, but    possibly unbounded    on $[0,T_x)$, in short  $x\in BV_{loc}[0,T_x)$, could be given  (see  \cite{KDPS14} and also \cite{KDPS15,AR15,MS18,MS20}).

\begin{proposition}\label{p3}
Assume  $f$  continuous and satisfying {\bf (Hg)}. If the rescaled system \eqref{Er} is GAC to $\C$,  then the extended system \eqref{Es} is GAC to $\C$.
\end{proposition}
\begin{proof}
Since the rescaled control system \eqref{Er} is GAC to $\C$,  there exists a descent rate $\beta\in \mathcal{KL}$ such that  for all $z\in \R^n\setminus \C$   there is an admissible rescaled trajectory-control pair $(y,v)$, such that $y(0)=x$ and 
\begin{equation}\label{eqex}
\d(y(s))\leq \beta(\d(z),s)\quad \forall  s\geq 0.\end{equation} 
Define for a.e.  $s\in[0,S_y)$ the extended control  
$$(w_0,w)(s)= \left(\frac{1}{1+\nu(|v(s)|)}\, , \,  \frac{v(s)\,\nu(|v(s)|)}{|v(s)|(1+\nu(|v(s)|))} \right).$$
Since $(w_0,w)(s)\in \U$  for  a.e. $s\in[0,S_y)$, from Lemma \ref{L1ext} it follows that  $(y,w_0,w)$ is an admissible extended trajectory-control pair for \eqref{Es} and this, together with \eqref{eqex} and the arbitrariness of $z$,  implies that the extended system  \eqref{Es} is GAC to $\C$,  with the same descent rate $\beta$ as the rescaled system \eqref{Er}.  
\end{proof}


\section{A Converse  Lyapunov Theorem}\label{s4}
In this section we state our main result, which  extends to control systems with unbounded controls and   their impulsive extensions well-known relationships between GAC, sample stabilizability, and  existence of control Lyapunov functions.  Furthermore, we relate explicit stabilizing feedback constructions for the extended  and for the original system,  which are based on the existence of a {\em semiconcave} control Lyapunov function. 

\vskip 0.2 truecm
\subsection{Main result}
To begin with, let us introduce  the notion of (nonsmooth) control Lyapunov function. In the following, given an open set $\Omega\subseteq\R^N$, a continuous function $W:\overline{\Omega} \to[0,+\infty)$ is said {\em positive definite on $\Omega$} if  $W(x)>0$ \,$\forall x\in\Omega$ and $W(x)=0$ \,$\forall x\in\partial\Omega$. The function $W$ is called {\em proper  on $\Omega$}  if the pre-image $W^{-1}(K)$ of any compact set $K\subset[0,+\infty)$ is compact.   
  As customary,    $\partial_PW(x)$  refers to the {\em proximal subdifferential of $W$ at $x$} (which may very well be empty). We recall that $p$ belongs to $\partial_PW(x)$  if and only if there exist $\sigma$ and $\eta>0$ such that
$$
W(y)-W(x)+\sigma|y-x|^2\ge \langle p\,,\, y-x\rangle \qquad \forall y\in  B_\eta(\{x\}).\C
$$
The {\em limiting subdifferential  $\partial_LW(x)$ of $W$ at $x\in \Omega$,}  is defined as
	$$
	\displaystyle \partial_LW(x) := \Big\{\lim_{i\to+\infty} \,  p_i: \  p_i\in \partial_PW(x_i), \ \lim_{i\to+\infty} x_i=x\Big\}.
	$$
 When the function $W$ is locally Lipschitz continuous on $\Omega$, the limiting subdifferential $\partial_LW(x)$ is nonempty at every point, the set-valued map $x\rightsquigarrow \partial_LW(x)$ is upper semicontinuous, and the  Clarke generalized gradient at $x$  can be derived as co\,$\partial_LW(x)$.  As sources for  nonsmooth analysis  we refer e.g. to \cite {CS,CLSW,Vinter}.
 \vskip 0.2 truecm
  For any nonempty closed set  $\mathbf U\subseteq \R^M$ for some integer $M>0$ and any continuous function   $\mathbf f:(\R^n\setminus\C)\times \mathbf U\to\R^n$,  let  consider the control system
\begin{equation}\label{Egenex}
\dot x=\mathbf f(x,u),  \qquad u\in  \mathbf U, 
\end{equation}
and   the {\em Hamiltonian}  $H_{\mathbf f,\mathbf U}:(\R^n\setminus\C)\times U\to[-\infty,+\infty)$, given by
	\begin{equation}\label{Ham}
	\displaystyle H_{\mathbf f,\mathbf U}(x,p):=\inf_{u\in \mathbf U}\left\{\langle p\,,\,\mathbf f(x,u)\rangle\right\}.
	\end{equation} \
	Notice that  $H_{\mathbf f,\mathbf U}$ may be discontinuous and   equal to $-\infty$ at some points. 	
	
	\begin{definition}[Control Lyapunov Function]\label{defCLF} Let $W:\overline{\R^n\setminus {\C}}\to[0,+\infty)$ be a  locally Lipschitz continuous function on $\overline{\R^n\setminus {\C}}$,  which is positive definite  and  proper on $\R^n\setminus\C$. We say that $W$  is a \emph{Control Lyapunov Function,}  (CLF),  {for the system \eqref{Egenex}} if there exists some continuous, strictly increasing function $\gamma:(0,+\infty)\to(0,+\infty)$, that we call a {\em decrease rate}, such that the following {\emph{(infinitesimal) decrease condition}} is satisfied:
		\begin{equation}\label{CLF} 
		H_{\mathbf f,\mathbf U} (x,  \partial_L W(x) )<-\gamma(W(x)) \quad \forall x\in {{\R^n}\setminus\C}.  \,\footnote{This  means that $H_{\mathbf f,\mathbf U}(x, p )<-\gamma(W(x))$ for every $p\in {\partial_L}W(x)$.} 
		\end{equation}
	\end{definition}
	
 \begin{remark}\label{RpartialP} {\rm  If the continuous function  $\mathbf f:  ( \R^n\setminus \C)\times \mathbf U\to \R^n$ is bounded in $(B_R(\C)\setminus\C)\times  \mathbf U$   for some $R>0$  and continuous in $x$ uniformly with respect to $  \mathbf U$ --as it is for the rescaled dynamics $\f$ and for the extended dynamics $F$--,  then the Hamiltonian $H_{\mathbf f,\mathbf U}$ is continuous and   the decrease condition \eqref{CLF} is equivalent to the usual condition
		\begin{equation}\label{CLFP} 
		H_{\mathbf f,\mathbf U} (x, \partial_PW(x) )<- V(x) \quad \forall x\in {{\R^n}\setminus\C},
		\end{equation}
 expressed in terms of the proximal subdifferential,  for some continuous function  $V:\overline{\R^n\setminus {\C}}\to[0,+\infty)$, which is positive definite  and  proper on $\R^n\setminus\C$,   used  e.g. in  \cite{CLSS,CLRS,KT04} (see \cite[Prop. 4.2]{LM18}). Incidentally, in this case \eqref{CLF} has also an equivalent formulation, which involves the  Dini derivative.}
\end{remark} 

 The rescaled system and the extended system   share the same CLFs.
 \begin{proposition}\label{p2} Assume   $f:(\R^n\setminus\C)\times U\to \R^n$ continuous and satisfying  assumption {\bf (Hg)}. A  map $W:\R^n\setminus\C\to\R$ is a CLF   for the rescaled problem \eqref{Er} if and only if it is a CLF for the extended problem \eqref{Es}.
\end{proposition} 
\begin{proof}
Consider the maps
\begin{equation}\label{eqresex}
\begin{aligned}
\displaystyle u&\mapsto (w_0,w)(u):=  \left(\frac{1}{1+\nu(|u|)}\, , \,  \frac{u\,\nu(|u|)}{|u|(1+\nu(|u|))} \right) \quad \forall u\in U, \\
\displaystyle (w_0,w)&\mapsto u(w_0,w):= \frac{w}{|w|}\,\nu^{-1}\left(\frac{|w|}{w_0}\right)  \quad \forall (w_0,w)\in  \U,
\end{aligned}
\end{equation}
where $\U$ and $\nu$ are as in   {\bf (Hg)}. The definitions of $\f$ and $F$   imply that  
$$
\begin{array}{l}
\displaystyle F(x,(w_0,w)(u))=\f(x,u)\qquad \forall (x,u)\in (\R^n\setminus\C)\times U, \\ 
\displaystyle \f(x,u(w_0,w))=F(x,w_0,w)  \ \ \forall (x,w_0,w)\in (\R^n\setminus\C)\times \U. 
\end{array}
$$
Thus, for any $x\in\R^n\setminus\C$ and   $p\in \R^n$, one has  
$$
\begin{array}{l} 
\displaystyle \inf_{u\in U}\left\{\langle p\,,\,\f(x,u)\rangle\right\}=\inf_{u\in U}\left\{\langle p\,,\,F(x,(w_0,w)(u))\rangle\right\} \\
\displaystyle\qquad\ \qquad\quad\qquad\qquad\ge \inf_{(w_0,w)\in   \overline{\U}}\left\{\langle p\,,\,F(x,w_0,w)\rangle\right\} \\
\displaystyle\quad\qquad\qquad\qquad\qquad=\inf_{(w_0,w)\in   \U}\left\{\langle p\,,\,F(x,w_0,w)\rangle\right\}\\
\displaystyle\quad\qquad\qquad\qquad\qquad=\inf_{(w_0,w)\in   \U }\left\{\langle p\,,\,\f(x,u(w_0,w))\rangle\right\} \\
\displaystyle\qquad\ \qquad\quad\qquad\qquad\ge \inf_{u\in U}\left\{\langle p\,,\,\f(x,u)\rangle\right\}.
\end{array}
$$
Therefore $H_{\f,U}\equiv H_{F, \,\overline{\U}}$. 
As a consequence,  a map $W$ is a CLF for  \eqref{E} if and only if it is a CLF for \eqref{Es}.   
\end{proof}

We are   ready to state the main result of the paper. To this aim, we introduce the following  stronger assumptions.
 \vskip 0.2 truecm
\noindent   {\bf (Hg)$^*$} {\em The function $f:\R^n\times U\to \R^n$ is continuous  and a stronger version of hypothesis  {\bf (Hg)} is valid, where $\R^n$ replaces $\R^n\setminus\C$. Furthermore,   for any compact set $\mathcal{K}\subset\R^n$ there is some constant  $\bar L>0$ such that
$$
|F(x,w_0,w)-F(y,w_0,w)|\le L|x-y| \qquad \forall x,y\in\mathcal{K}, \ \forall (w_0,w)\in\bar\U. 
$$ }
\begin{theorem}[Converse Lyapunov Theorem]\label{thm3}
Assume hypothesis {\bf (Hg)$^*$}. Then the following properties are equivalent:
\begin{itemize}
\item[{\rm (i)}] the original control system  $\dot x=f(x,u)$ is GAC to $\C$;
\item[{\rm (ii)}] the extended control system $y'=F(y,w_0,w)$ is GAC to $\C$;
\item[{\rm (iii)}] there exists a CLF for  the extended control system  $y'=F(y,w_0,w)$;
\item[{\rm (iv)}] there exists a CLF for  the original control system  $\dot x=f(x,u)$;
\item[{\rm (v)}]   the system  $\dot x=f(x,u)$ is sample stabilizable to $\C$;
\item[{\rm (vi)}]  the  system $y'=F(y,w_0,w)$ is sample stabilizable to $\C$.
\end{itemize}
 \end{theorem}
 
\begin{proof} Let us preliminarily observe that  the Lipschitz continuity hypothesis on $F$ in  {\bf (Hg)$^*$} implies both that $f$ satisfies hypothesis {\bf (Hl)}  and   that   $\f$ is locally Lipschitz continuous in $x$ (on $\overline{\R^n\setminus\C}$), uniformly w.r.t. $u\in U$.
\vskip 0.2 truecm
${\rm (i)}\Longrightarrow {\rm (ii)}$.  From Theorem \ref{ThEEr} it follows that the original control system  is GAC to $\C$ if and  only if the rescaled control system $y'=\f(y,v)$ is GAC to $\C$. This implies that  the extended control system   is GAC to $\C$,  in view of Proposition  \ref{p3}.
\vskip 0.2 truecm
  ${\rm (ii)}\Longleftrightarrow {\rm (iii)}$. The fact that the extended control system is GAC to $\C$ if and only if there exists a CLF for it follows from  \cite[Theorem 1]{KT00} (see also \cite[Thm 3.2, Rmk. 4]{KT04}). Observe that this result is applicable to the impulsive extension  essentially because the set $\overline{\U}$ of extended control values is bounded. 
\vskip 0.2 truecm
  ${\rm (iii)}\Longrightarrow {\rm (iv)}$. Let $W$ be a CLF for the  extended control system  $y'=F(y,w_0,w)$, for some decrease rate $\gamma$.  Hence, by Proposition \ref{p2}  $W$ is also a CLF for the rescaled system $y'=\f(y,v)$, with  the same $\gamma$. Since the definition of $\bar f$  implies that 
\begin{equation}\label{orrs} 
H_{f,U}(x,p)\leq H_{\bar f,U}(x,p)< -\gamma(W(x))\quad \forall x\in \R^n\setminus \C, \ \ \forall p\in \partial_L W(x), 
\end{equation}
we can finally conclude that  $W$ is  a CLF also for the original control system. 
\vskip 0.2 truecm
 ${\rm (iv)}\Longrightarrow {\rm (v)}$. Let  $W$ be a CLF for \eqref{E}. Then,  \cite[Theorem 4.6]{LM20} implies that $\dot x=f(x,u)$ is sample stabilizable  to $\C$.
\vskip 0.2 truecm
 ${\rm (v)}\Longrightarrow {\rm (i)}$. The sample stabilizability of $\dot x=f(x,u)$ to $\C$ implies that it is GAC to $\C$,  by  Theorem   \ref{PSGAC}.  
 \vskip 0.2 truecm  
With this,  we have shown that (i),(ii),(iii),(iv), and  (v) are equivalent. To conclude the proof it suffices to observe that, by the same arguments as above,  ${\rm (iii)}\Longrightarrow {\rm (vi)}$ and ${\rm (vi)}\Longrightarrow {\rm (ii)}$, namely,  the existence of a CLF   for $y'=F(y,w_0,w)$ implies sample stabilizability of the extended system, which in turn implies that $y'=F(y,w_0,w)$ is GAC to $\C$.  
 \end{proof}

\subsection{Semiconcave CLFs and stabilizing  feedback construction} 
Given an  open  set   $\Omega\subseteq\R^N$, a   function $W:\Omega \to\R$  is  called {\em  locally semiconcave} if for every compact  subset $\mathcal K\subset \Omega$  there exists $\rho>0$ such that, for all   $x$, $\hat x\in \mathcal K$ with $[x,\hat x]\subset\mathcal K$, one has
	$$
	W(x)+W(\hat x)-2W\left(\frac{x+\hat x}{2}\right)\le \rho|x-\hat x|^2.
	$$
	Locally semiconcave functions are locally Lipschitz continuous and   twice differentiable almost everywhere  (see e.g. \cite{CS}).

\vskip 0.2 truecm
 Since the works by Rifford \cite{R00,R02},  semiconcave  control Lyapunov functions  have proven to   play a crucial role for the explicit construction of sample stabilizing feedback strategies.  It is therefore worth  noting that
in Theorem \ref{thm3}  we can   assume without loss of generality  that  CLFs  are  locally semiconcave on $\R^n\setminus\C$. Precisely, one has:
\begin{proposition}\label{psc}    
Under the assumptions of Theorem  \ref{thm3},
\footnote{Actually, from the results in \cite{LM18,LM20} this statement is valid even if $f$, $\f$, and $F$ satisfy the assumptions in {\bf (Hg)$^*$} only for  $x\in \overline{\R^n\setminus\C}$.}  if there exists a CLF either for the extended control system  \eqref{Es} or for the original control system  \eqref{E}, then there exists a CLF for \eqref{Es} or  \eqref{E}, respectively, which is locally semiconcave on $\R^n\setminus\C$.
\end{proposition} 
\begin{proof} If there exists a locally Lipschitz continuous  CLF  for the extended system \eqref{Es}, the dynamics of which meet classical   Lipschitz continuity and boundedness assumptions, then from   \cite[Theorem 4.3]{LM18}  there  is also a   locally semiconcave CLF  for \eqref{Es}. On the other hand, when there exists  a locally Lipschitz continuous CLF    for the original system \eqref{E},   \cite[Theorem 4.3]{LM20}   guarantees the existence of  a locally Lipschitz continuous CLF for the rescaled system \eqref{Er}. At this point,   \cite[Theorem 4.3]{LM18} again implies the existence of a   locally semiconcave CLF  for \eqref{Er}, which, as one  deduces by \eqref{orrs},   is also a  locally semiconcave CLF  for \eqref{E}. The proof of the claim is thus complete. 
\end{proof} 

Thanks to Proposition \ref{psc}, we can explicitly build a stabilizing feedback for the original control  system from a stabilizing feedback for the impulsive extension.  To this aim, we need the following preliminary result.  
  \begin{proposition}\label{pN}  Assume that $F:(\R^n\setminus\C)\times\overline{\U}\to\R$ is a continuous function.
Let $W$ be a CLF for   $y'=F(y,w_0,w)$, with decrease rate $\gamma$. Then, there exists  a continuous function $N:(0,+\infty)\to (0,1]$ such that
\begin{equation}\label{CLFN} H_{  F, \U_{N(W(x))}}(x,\partial_L W(x))<-\gamma(W(x)) \quad \forall x\in \R^n\setminus \C,
\end{equation}
where, for every $\rho\in(0,1]$, 
$$\U_\rho:=\{(w_0,w)\in \U: \  w_0 \ge\rho\}.$$  
\end{proposition}

\begin{proof}
To prove the statement, we show that  there is some continuous function $N:(0,+\infty)\to (0,1]$, which is  increasing in $(0,1]$,  decreasing in $[1,+\infty)$,  and such that, for any $r>0$,  one has
\begin{equation}\label{sfj} H_{  F, \U_{N(r)}}(x,\partial_L W(x))<-\gamma(W(x)) \quad \forall x\in W^{-1}([ r\land 1, r\vee 1]).\end{equation}
In fact, proven this,  from the monotonicity properties  of $N$  it immediately follows that  in \eqref{sfj} we can replace  $r$ with $W(x)$, and that implies \eqref{CLFN} by the arbitrariness of $r>0$. 

To prove \eqref{sfj}, fix $r>0$ and set
$$
\Gamma_r:=\{(x,p): \ \ x\in W^{-1}([ r\land 1, r\vee 1]), \ p\in\partial_LW(x)\}.
$$
Notice that the  properties of $W$ --in particular, the properness of $W$ and  the upper semicontinuity of the set-valued map $x\rightsquigarrow \partial_LW(x)$-- imply that $\Gamma_r$ is a compact set.
For every $(x,p) \in \Gamma_r$,  define
$$w_0(x,p):=\sup\{w_0: \  (w_0,w)\in \U \ \text{and} \  \langle p, F( x,w_0,w)\rangle<-\gamma(W( x))\}. $$
 This set is nonempty because $W$ is a CLF, and  $w_0(x,p)\in[0,1]$. At this point, set
$$\hat N(r):=\inf\{w_0(x,p): \ (x,p)\in\Gamma_r\}.$$
By construction,  $\hat N$ is nonnegative, increasing in $(0,1]$ and decreasing in $[1,+\infty)$,  and $\hat N([0,+\infty))\subseteq[0,1]$. If $\hat N(r)>0$ for all $r>0$, then the required $N$ is given by any continuous, positive approximation from below of $\hat N$, which is increasing in $(0,1]$ and  decreasing in $[1,+\infty)$.
\noindent To conclude,  it only remains to prove that $\hat N(r)>0$ for all $r>0$. To this end, assume by contradiction that $\hat N(r)=0$ for some $r>0$. Then, there is some sequence $((x_k,p_k))_{k\ge 1}\subset \Gamma_r$, such that $w_0(x_k,p_k)<1/k$ for all $k$.  Hence,  the definition of $w_0(x_k,p_k)$ yields  
\begin{equation}\label{contreq0}  \langle p_k,F(x_k,w_0,w)\rangle<-\gamma(W(x_k)) \ \text{ for } \ (w_0,w)\in \overline{\U} \Longrightarrow w_0<\frac{1}{k}.\end{equation}
Since $\Gamma_r$ is compact,  there exists a subsequence, that we still denote $((x_k,p_k))_k$, converging to some $(\bar x,\bar p)\in\Gamma_r$. Since $W$ is a CLF with decrease rate $\gamma$,   there exists some $(\bar w_0,\bar w)\in \U$ (with $\bar w_0>0$) such that
$$
 \langle \bar p, F(\bar x,\bar w_0,\bar w) \rangle<-\gamma(W(\bar x)).$$
 Therefore, by the continuity of $F,W$ and $\gamma$,  for a sufficiently large $k$  one has $1/k<\bar \omega_0$ and 
$$
 \langle p_k, F( x_k,\bar w_0,\bar w)\rangle<-\gamma(W(x_k)),$$
 in contradiction with \eqref{contreq0}, so that  the proof is complete. 
\end{proof}

From Proposition \ref{pN} it follows that, given a semiconcave control Lyapunov function for the  extended control system, we can always select a stabilizing feedback $\hat K(x)=(\hat w_0(x),\hat w(x))$  {\em which is not impulsive,} namely such that $\hat w_0(x)>0$ for every $x\in\R^n\setminus\C$.
\begin{proposition}\label{pK} Consider the same assumptions as in  Theorem  \ref{thm3}. Let  the extended system \eqref{Es} be sample stabilizable to $\C$.  
Then, there exist a continuous function $N:(0,+\infty)\to(0,1)$ and a stabilizing feedback  $\hat K: \R^n\setminus \C\to \U$, $\hat K(x)=(\hat w_0(x),\hat w(x))$ for  \eqref{Es},  satisfying
\begin{equation}\label{w0N}
w_0(x)\geq N(W(x))\qquad  \forall x\in \R^n\setminus \C,
\end{equation}
  so that
 the locally bounded feedback $K: \R^n\setminus \C\to U$
given by
\begin{equation}\label{KN}
K(x):=\frac{\hat w(x)}{|\hat w(x)|}\, \nu^{-1}\left({\frac{|\hat w(x)|}{\hat w_0(x)}}\right) \qquad  \forall x\in \R^n\setminus \C
\end{equation}
is sample stabilizing for the original system \eqref{E} to $\C$.
\end{proposition}
\begin{proof}
From  Theorem \ref{thm3} and Proposition \ref{psc} it follows that,  if the extended control system $y'=F(y,w_0,w)$ is sample stabilizable to $\C$, then it admits a locally semiconcave CLF $W$. Hence, Proposition \ref{pN} with reference to $W$ implies the existence of a continuous function  $N:(0,+\infty)\to (0,1]$  such that, fixed  a selection $p(x)\in \partial_LW(x)$ for any  $x\in \R^n\setminus\C$,  any map  $\hat K:\R^n\setminus\C\to \U$ such that
 $$
 \hat K(x)=(\hat w_0(x),\hat w(x))\in \underset{(w_0,w)\in \U_{N(W(x))}}{\arg\min}\Big\{\langle p(x),F(x,w_0,w)\rangle  \Big\},
 $$ 
 satisfies \eqref{w0N}  and the inequality
$$
\langle p(x),F(x,w_0,w)\rangle <-\gamma(W(x)) \qquad\forall x\in\R^n\setminus\C. 
$$
As shown in  \cite[Sect. 3]{LM20}, this implies that $\hat K$  is a  sample stabilizing feedback for  the extended system \eqref{Es}. 
Consider now the feedback $K: \R^n\setminus \C\to U$ associated to such $\hat K$, defined  as in \eqref{KN}. It is locally bounded, since   
$$
|K(x)|\le \nu^{-1}\left({\frac{1}{N(W(x))}}\right) \qquad \forall x\in\R^n\setminus\C.
$$
Furthermore, $K$ is  sample stabilizing for the rescaled control system $y'=\bar f(y,u)$, in view of  the identity $\bar f(x,K(x))=F(x,\hat K(x))$ for all $x\in \R^n\setminus \C$, from which it follows that  sampling trajectories associated to $\hat K$ for the extended system coincide with the sampling trajectories associated to $K$ for the rescaled system. Since by  \cite[Theorem 2.5]{LM20}  the rescaled and the original system share the same stabilizing feedbacks,   $K$ is  sample stabilizing also for  $\dot x=f(x,u)$ and   the proof is concluded.
\end{proof}

\begin{remark} {\rm Observe that the converse relation, namely the fact that, given a stabilizing feedback   $K$ for the original system \eqref{E}, it is possible to derive a feedback $\hat K$ for the impulsive extension, is quite obvious. Indeed, it is easy to see that  the map 
$$\hat K: x\mapsto\left(\frac{1}{1+\nu(|K(x)|)}\, , \,  \frac{K(x)\,\nu(|K(x)|)}{|K(x)|(1+\nu(|K(x)|))} \right) \quad\forall x\in\R^n\setminus\C$$
is a stabilizing feedback for the extended system \eqref{Es}, again by the identity $F(x,\hat K(x))=\bar f(x, K(x))$ together with  \cite[Theorem 2.5]{LM20}.}
\end{remark}

\section{An example}\label{Sex} 
 In this  section, we introduce a simple control system which is neither globally asymptotically controllable nor sample stabilizable to the origin by means of bounded controls,  whereas it is by means of unbounded   strategies.  Furthermore, we show how to construct a stabilizing feedback for the original system, given a control Lyapunov function (and an associated  stabilizing feedback)  of the extend system. 
 
\vskip 0.2 truecm
Consider the  target  $\T:=\{0\}$ and the    one-dimensional control system
\begin{equation}\label{ex1}
\dot x(t)=x(t)-x^3(t)u(t), \quad u(t)\in U:=[0,+\infty) \quad\text{a.e}.
\end{equation}
Using only  controls taking  values in a  bounded subset $[0,M]$ of $U$ for some $M>0$, the best strategy to approach the origin is clearly to implement the constant control $u\equiv M$ and  solve the differential equation
$$\dot x(t)=x(t)-Mx^3(t).$$
But this way, for every initial point $z\ne0$ we get the trajectory 
\begin{equation}
\label{exsol}
x(t)=\frac{ze^t}{\sqrt{z^2M(e^{2t}-1)+1}},
\end{equation}   
so we have   $\displaystyle\lim_{t\to+\infty}x(t)=\frac{\text{sign}(z)}{\sqrt{M}}\ne 0$. Therefore,  the control system \eqref{ex1} with   controls in any given bounded subset of  $U$  is not GAC to $\{0\}$. 

In view of Theorem \ref{thm3},  the global asymptotic controllability and the sample stabilizability of system \eqref{ex1} to $\{0\}$  when admissible pairs $(x,u)$ with controls $u\in L^\infty_{loc}([0,T_x),[0,+\infty))$ are allowed, is equivalent to the global asymptotic controllability to $\{0\}$  of  the impulsive extension 
\begin{equation}\label{ex1es}
y'(s)=F(y(s),w_0(s),w(s))=y(s)\,w_0(s)-y^3(s)w(s), \quad (w_0,w)(s)\in \overline{\U}  \quad\text{a.e.},
\end{equation}
where $\U:=\{(w_0,w)\in(0,+\infty)\times[0,+\infty), \ \ w_0+w=1\}$. Here, by choosing the  constant control $(w_0,w)(s)\equiv(0,1)$ for every $s\ge 0$,   for each starting point  $z\ne0$  we get an extended trajectory (describing  in the original time variable  an instantaneous jump  from $z$ to the target) that satisfies
$$\d(y(s))=|y(s)|=\frac{1}{\sqrt{2s+\frac{1}{z^2}}}=:\beta(|z|,s)\qquad \forall s\geq0,
$$ 
where $\beta\in{\mathcal {KL}}$. So, the original system is GAC and sample stabilizable to the origin,  because  the extended system is GAC to $\{0\}$.   

Again Theorem \ref{thm3} together with Propositions \ref{psc}, \ref{pK}  guarantees that there is a locally semiconcave control Lyapunov function for 
the  extended  system, which makes it possible to build both a sample stabilizing feedback $\hat K$ for \eqref{ex1es} and    a locally bounded sample stabilizing feedback $K$ for \eqref{ex1}. In particular,  a locally semiconcave CLF for {\eqref{ex1es}} is given by the function $W(x):=|x|$ for all $x\in\R$.  Indeed,  for every $x\ne0$,  one has
$$H_{F,\U}(x,\partial_LW(x))=\inf_{(w_0,w)\in\U} \left\{\frac{x}{|x|}\,(xw_0-x^3w)\right\}
=-|x|^3<-\gamma(W(x)),$$
if we choose  the decrease rate   $\gamma(r):=\frac{r^3}{2(2+r^2)}$, $r>0$.  At this point, setting $N(r):= \frac{r^2}{2+r^2}$, we can define the feedback $\hat K:\R\setminus\{0\}\to\U$, given by 
$$
\hat K(x)=(\hat w_0(x),\hat w(x)):= (N(|x|), 1-N(|x|))=\left(\frac{x^2}{2+x^2}, \frac{2}{2+x^2}\right), 
$$ 
which is sample stabilizing for the extend system, since   
$$
\frac{x}{|x|}\,(x\hat w_0(x)-x^3\hat w(x))= \frac{x}{|x|}\left(x\frac{x^2}{2+x^2}-x^3 \frac{2}{2+x^2}\right)=-\frac{|x|^3}{2+x^2}<-\gamma(W(x)),
$$
 for any $x\ne0$. At this point, from Proposition \ref{pK} it follows that  the locally bounded  feedback $K:\R\setminus\{0\} \to U$, defined by
$$
K(x)=\frac{\hat w(x)}{\hat w_0(x)}= \frac{2}{x^2} \qquad\forall x\ne0,
$$
is sample stabilizing for the original  control system \eqref{ex1}. 
In particular, an  associated  descent rate  is  $\beta(R,t):=R e^{-t/2}$ for all $(R,t)\in[0,+\infty)^2$. Indeed, let $(R,r)$ be a pair  with $0<r<R$, and choose any  sampling time  $\delta(R,r)>0$, which is  continuous, $r$-increasing and $R$-decreasing,   such that $\delta(R,r)\leq  \ln(\varphi)$, where $\varphi:=(1+\sqrt{5})/2$ is the Golden Mean. Then, for any partition  $\pi=(t_i)_i$    of $[0,+\infty)$ with sampling time $\delta(R,r)$,  the  $\pi$-sampling trajectory from each $z\not=0$ with $\d(z)\le R$   associated to $K$,  satisfies the recursive relation
$$
\dot x(t)=x(t)-\frac{1}{(x(t_n))^2} x^3(t) \quad t\in [t_n, t_{n+1}],\quad x(0)=z.$$
In view of the definition of $\delta(R,r)$, one has $e^{t-t_n}\in[1,\varphi]$ for all $t\in[t_n,t_{n+1}]$ and, in particular, this implies $e^{3(t-t_{n+1})}-2e^{2(t-t_n)}+1\leq 0$ for all $t\in[t_n,t_{n+1}]$. } Using \eqref{exsol} (with $M=2(x(t_n))^{-2}$)
we then obtain, after few computations, the estimate
$$|x(t)|=|x(t_n)|\frac{e^{t-t_n}}{\sqrt{2e^{2(t-t_n)}-1}}\leq |x(t_n)| e^{-(t-t_n)/2} \quad \forall t\in[t_n,t_{n+1}].$$
In particular $|x(t_{n+1})|\leq |x(t_n)|e^{-(t_{n+1}-t_n)/2}$, for all $n\geq 0$, therefore
$$|x(t)|\leq |z| e^{-t/2}=\beta(|z|,t)\qquad \forall t\geq 0.$$

\bibliographystyle{alpha}
\bibliography{lyap}




%

\end{document}